\documentclass[11pt,a4paper,reqno]{amsart}     

\usepackage{mathptmx}      
\usepackage{amsmath,amssymb,hyperref}

\newtheorem{theorem}{Theorem}
\numberwithin{theorem}{section}
\newtheorem{proposition}{Proposition}
\numberwithin{proposition}{section}
\newtheorem{lemma}{Lemma}
\numberwithin{lemma}{section}
\newtheorem{corollary}{Corollary}
\numberwithin{corollary}{section}

\theoremstyle{definition}
\newtheorem{example}{Example}
\numberwithin{example}{section}

\theoremstyle{remark}
\newtheorem{remark}{Remark}
\numberwithin{remark}{section}

\providecommand{\BBb}[1]{{\mathbb{#1}}}

\providecommand{\cal}[1]{{\mathcal{#1}}}

\newcommand{\Bbar}{\overline{B}}
\newcommand{\C}{{\BBb C}}
\newcommand{\dual}[2]{\langle\,#1,\,#2\,\rangle}

\newcommand{\fracc}[2]{{
                \textstyle\frac{#1}{\raise 1pt\hbox{$\scriptstyle #2$}}}}

\newcommand{\fracnp}{\fracc np}
\newcommand{\fracci}[2]{{\frac{#1}{\raise 1pt\hbox{$\scriptscriptstyle #2$}}}}
\newcommand{\fracpi}{\fracci1p}

\newcommand{\im}{\operatorname{i}}
\newcommand{\loc}{\operatorname{loc}}

\newcommand{\mlap}{-\!\operatorname{\Delta}}
\newcommand{\nrm}[2]{\|#1\|_{#2}}
\newcommand{\Nrm}[2]{\bigl\|#1\bigr\|_{#2}}

\newcommand{\order}{\operatorname{order}}
\newcommand{\op}[1]{\operatorname{#1}}
\newcommand{\OP}{\operatorname{OP}}
\newcommand{\N}{\BBb N}
\newcommand{\R}{{\BBb R}}
\newcommand{\Rn}{{\BBb R}^{n}}
\newcommand{\supp}{\operatorname{supp}}

\renewcommand{\hat}[1]{\overset{{\scriptscriptstyle \wedge}}{#1}}

\begin{document}

\title{{Pointwise estimates of pseudo-differential operators}}
\author{Jon Johnsen}
\thanks{Supported by 
the Danish Council for Independent Research, Natural Sciences (Grant No. 09-065927)%
\\[4\jot]
{\tt Appeared in Journal of pseudo-differential operators and their applications,
{\bf 2\/} (2011), pp.~377--398}%
}

\address{Departments of Mathematical Sciences, Aalborg University,
Fredrik Bajers Vej 7G,
DK-9220 Aalborg {\O}st, Denmark}
\email{jjohnsen@math.aau.dk}           

\begin{abstract}
As a new technique it is shown how general pseudo-differential operators 
can be estimated at
arbitrary points in Euclidean space when acting on functions $u$ with
compact spectra. The estimate is a factorisation inequality, in
which one factor is the Peetre--Fefferman--Stein maximal function of $u$,
whilst the other is a symbol factor carrying the whole information on the
symbol.
The symbol factor is estimated in terms of the spectral radius of $u$, so that
the framework is well suited for Littlewood--Paley analysis.
It is also shown how it gives easy access to results on polynomial bounds
and estimates in $L_p$, including a new result for type $1,1$-operators that
they are always bounded on $L_p$-functions with compact spectra.
\end{abstract}
\enlargethispage{2\baselineskip}\thispagestyle{empty}
\maketitle

\section{Introduction}  \label{intr-sect}
The aim of this note is to show how one can estimate 
a pseudo-differential operator at an arbitrary point $x\in \Rn$.
These pointwise estimates are applied to mapping properties and continuity results, 
in order to illustrate their efficacy.

\bigskip

The central theme is to show for a general symbol 
$a(x,\eta)$, with the
associated operator $a(x,D)u(x)=(2\pi)^{-n}\int e^{\im
x\cdot\eta}a(x,\eta)\hat u(\eta)\,d\eta$, 
that for distributions with compact spectra, i.e.\ $u\in {\cal F}^{-1}{\cal E}'(\Rn)$,
\begin{equation}
  |a(x,D)u(x)|\le c\cdot u^*(x) \quad\text{for every}\quad x\in \Rn.
  \label{au*-eq}
\end{equation}
Here $u^*$ denotes the Peetre-Fefferman-Stein maximal function
\begin{equation}
  u^*(x)=u^*(N,R;x)=\sup_y\frac{|u(x-y)|}{(1+|Ry|)^N}
        =\sup_y\frac{|u(y)|}{(1+R|x-y|)^N},
  \label{u*-eq}
\end{equation}
where $N>0$, $R>0$ are parameters;
$R$ so large that
$x\in \supp\hat u$ implies $|x|\le R$. 

One obvious advantage of proving \eqref{au*-eq} in terms of 
\eqref{u*-eq} is the \emph{immediate} $L_p$-estimate
\begin{equation}
  \int |a(x,D)u(x)|^p\,dx \le c^p\int |u^*(x)|^p\,dx \le c^pC_p\nrm{u}{p}^p,
 \qquad 1\le  p\le\infty,
  \label{iLp-eq}
\end{equation}
where the last step is to invoke the \emph{maximal inequality} 
\begin{equation}
  \int_{\Rn} |u^*(x)|^p\,dx \le C_p\int_{\Rn}|u(x)|^p\,dx,
  \qquad u\in L_p\cap{\cal F}^{-1}{\cal E}'.
  \label{max-ineq'}
\end{equation}
This estimate of the non-linear map $u\mapsto u^*$
has for $Np>n$ been known  since 1975 from a work of Peetre \cite{Pee75TL}, who
estimated $u^*(x)$ by the Hardy-Littlewood maximal function 
$Mu(x)=\sup_{\rho>0}c_n\rho^{-n}\int_{|y|<\rho} |u(x+y)|\,dy$
in order to invoke $L_p$-boundedness of the latter.
A significantly simpler proof is given below.

It is remarkable that little attention has been paid over
the decades to pointwise estimates like \eqref{au*-eq}\,---\,in comparison Peetre's
proof of \eqref{max-ineq'} quickly got a central role in the theory of function
spaces; cf.\ \cite[1.4.1]{T2}. 
However, to the
author's knowledge, there has only been a similar attempt by Marschall, who in
his thesis \cite{Mar85} suggested to estimate $a(x,D)u(x)$ in terms of $Mu$; this was
followed up in a series of papers, e.g.\ \cite{Mar91,Mar95,Mar96}, where the technique was used
to derive boundedness under weak assumptions in spaces based on $L_p$
(functions and symbols subject to Besov and Lizorkin--Triebel conditions).

In the present paper the point of view is quite different. First of all because $u^*$  is
rather easier to treat and work with than $Mu$. Secondly, the aim is to
explain how pointwise estimates in terms of $u^*$ will simplify well-known
topics such as $L_p$-estimates and Littlewood--Paley analysis of $a(x,D)$.

So as a main result here, \eqref{au*-eq} is also shown to be 
straightforward to obtain;
cf.\ Theorem~\ref{Fau*-thm}--\ref{Fa-thm} below. Indeed, the
constant $c$ in \eqref{au*-eq} is just an upper bound for the \emph{symbol
factor} $F_a(x)$, which is a continuous, bounded function carrying the entire
information of the symbol in the \emph{factorisation}
inequality 
\begin{equation}
  |a(x,D)u(x)|\le F_{a}(x) u^*(x),\qquad u\in{\cal F}^{-1}{\cal E}'(\Rn).
  \label{fact-ineq}
\end{equation}
As $F_a(x)$ only depends vaguely
on $u$ (cf.\ Section~\ref{pest-sect}), this gives a somewhat surprising decoupling. 

The inequality is well suited for Littlewood--Paley analysis of $a(x,D)$ as
described in Section~\ref{corona-sect}.
The set-up there has recently been exploited by the author 
\cite{JJ10tmp} in proofs of
fundamental results for pseudo-differential operators of type $1,1$; this is
briefly reviewed in Section~\ref{t11-sect}, where also \eqref{iLp-eq} is
given as a new theorem for type $1,1$-operators.

\section{The Peetre--Fefferman--Stein maximal function}   \label{PFS-sect}
This section explores the definition of $u^*(x)$ in \eqref{u*-eq}, in lack
of a reference. It also gives a straightforward proof of the maximal inequality
\eqref{max-ineq'}. 

\bigskip

For the reader's sake a few easy facts are recalled first.
To show that $u^*(x)$ is a `slowly' varying function, note that
\begin{equation}
  \frac{|u(x-z)|}{(1+|Rz|)^N}=
  \frac{|u(y-(z+y-x))|}{(1+R|z+y-x|)^N}\cdot
  \frac{(1+R|z+y-x|)^N}{(1+|Rz|)^N},
\end{equation}
so the inequality $1+|x+y|\le1+|x|+|y|+|x||y|=(1+|x|)(1+|y|)$ gives
\begin{equation}
  u^*(x)\le u^*(y)(1+R|x-y|)^N.
  \label{u*xy-ineq}
\end{equation}
Therefore  $u^*(x)$ is finite at every $x\in \Rn$ if it is so at one point $y$.
So either $u^*(x)=\infty $ on the entire $\Rn$, 
or \eqref{u*xy-ineq} implies that $u^*(x)$ is 
continuous on $\Rn$, i.e.\ $u^*\in C(\Rn)$. 

Finiteness is for large $N$ implied by the (often imposed)
assumption that $u\in {\cal S}'(\Rn)$ should have its
spectrum in the closed ball $\Bbar(0,R)$ of radius $R$, ie
\begin{equation}
  \supp{\cal F}u\subset \Bbar(0,R).
  \label{FuBR-eq}
\end{equation}
Indeed, then $|u(x)|\le c_R(1+R|x|)^m$ by the Paley--Wiener--Schwartz theorem, when
$m$ is the order of $\hat u$. So $N\ge m$ gives
$u^*(N,R;0)\le c_R$, hence $u^*(N,R;x)<\infty$ 
for all $x\in \Rn$.

In any case it is clear that $u\mapsto u^*$ is subadditive, i.e.\ $(u+v)^*\le u^*+v^*$, whence
\begin{equation}
  |u^*(N,R;x)-v^*(N,R;x)|\le (u-v)^*(N,R;x).
  \label{uvmu-eq}
\end{equation}
Hence $u\mapsto u^*$ is Lipschitz continuous on $L_\infty (\Rn)$ with
constant $1$, as
it is a shrinking map there, i.e.\ $\nrm{u^*}{\infty }\le \nrm{u}{\infty }$.
With respect to the H{\"o}lder seminorm
\begin{equation}
 |u|_\sigma:= \sup_{x\ne y}|u(x)-u(y)|/|x-y|^{\sigma},\qquad 0<\sigma<1 ,
\end{equation}
it is also a shrinking map, for \eqref{uvmu-eq} gives that
\begin{equation}
  |u^*(x+h)-u^*(x)|\le \sup_{\Rn} \frac{|u(x+h+\cdot )-u(x+\cdot )|}{(1+R|\cdot |)^N}
  \le |u|_\sigma |h|^\sigma.
\end{equation}
Therefore $|u^*|_\sigma\le |u|_\sigma$ as claimed. In particular one has

\begin{proposition}
  \label{Cs*-prop}
The map $u\mapsto u^*$ is for all $N, R>0$ a shrinking map on
the H{\"o}lder space $C^\sigma(\Rn)$, $0<\sigma< 1$, defined by 
finiteness of the norm $|u|_{\sigma}^*=\sup|u|+|u|_\sigma$.
\end{proposition}

A main case is when $u$ is in $ L_p(\Rn)$, $1\le
p\le\infty$. For $p<\infty $ one has that $u^*\equiv \infty $ for
$u$ equal to $e^{|x|}$ times the characteristic function of 
$\bigcup_{k\in\N} B(ke_1, e^{-(k+1)2p})$.
Such growth is impossible on the subspace of functions fulfilling the spectral
condition \eqref{FuBR-eq}, so this is imposed henceforth.

As an a priori analysis of this case, the Nikolski\u\i--Plancherel--Polya
inequality implies $u\in L_p\cap L_\infty$, for it states that if $u\in L_p$
and \eqref{FuBR-eq} holds, then 
\begin{equation}
  \nrm{u}{r}\le cR^{\tfrac{n}{p}-\tfrac{n}{r}}\nrm{u}{p}
  \quad\text{for}\quad p<r\le\infty.
  \label{urp-ineq}
\end{equation}
For its proof one can take 
an auxiliary function $\psi\in {\cal S}(\Rn)$ so that ${\cal F}\psi(\xi)=0$
for $|\xi|\ge2$ 
and ${\cal F} \psi(\xi)=1$ around $B(0,1)$, for then $u=R^n\psi(R\cdot)*u$,
and \eqref{urp-ineq} follows from this identity at once by
the Hausdorff--Young inequality $\nrm{f*g}{r}\le \nrm{f}{p}\nrm{g}{q}$,
where $\tfrac{1}{p}+\tfrac{1}{q}=1+\tfrac{1}{r}$; 
hereby $c=\nrm{\psi}{q}$, that only depends on $p$, $r$ and $n$.

To complete the picture, \eqref{urp-ineq} extends as it stands to the range
$0<p<r\le\infty$, provided  $u$ is given in $L_p\cap{\cal S}'(\Rn)$ with 
$\supp{\cal F}u\in \Bbar(0,R)$; cf.\ \cite[1.4.1(ii)]{T2}. 
The direct treatment in \cite{JoSi07} shows that one can take 
$c=\nrm{\psi}{\infty }^{\fracpi-\fracci1r}$ for $0<p\le 1$.
(For $0<p<1$, the set $L_p\cap {\cal S}'$ itself consists of the
$u\in L_1^{\loc}\cap {\cal S}'$ fulfilling $\int_{\Rn}|u|^p\,dx<\infty $, that
\emph{per se}
requires stricter smallness than $L_1$ for $|x|\to\infty $ but gives a global
condition on the singularities in the possibly non-compact region where
$|u(x)|>1$.) 

By \eqref{urp-ineq}, pointwise estimates of $u^*(x)$ hold for $L_p$-functions with compact spectra:

\begin{lemma}   \label{uu*-lem}
For every $u\in L_p\cap {\cal S}'(\Rn)$, $0<p\le
\infty $ with $\supp{\cal F}u\subset \Bbar(0,R)$, it holds true on $\Rn$ that
\begin{equation}
  |u(x)|\le u^*(N,R;x)
  \le \nrm{u}{\infty}<\infty  \quad\text{for every}\quad N>0.
  \label{uu*-ineq}
\end{equation}
\end{lemma}
\begin{proof}
With $r=\infty $ in \eqref{urp-ineq}
it follows that $\nrm{u}{\infty }$ is finite; and
it dominates $u^*(x)$ as stated, by the definition of $u^*$ in
\eqref{u*-eq}. Taking $y=0$ there yields $|u(x)|\le u^*(x)$.
\end{proof}

Note that ${\cal F}L_p\subset {\cal D}^{\prime\,k}$ for 
the least integer $k>\tfrac{n}{2}-\fracnp$ if $p>2$, 
cf.\ \cite[Sec.~7.9]{H},
so the Paley--Wiener--Schwartz theorem 
would give the poor condition
$N\ge [\tfrac{n}{2}-\fracnp]+1$ for finiteness of $u^*$.

For convenience in the following, the auxiliary function $f_N$ is
introduced as
\begin{equation}
 f_N(z)=(1+|z|)^{-N}.
\label{fuN-id}
\end{equation}

\begin{example}   \label{conv-exmp}
As is well known, $u^*$ is useful (when finite) for pointwise control of
convolutions, since e.g.\ the assumptions 
$\varphi\in {\cal S}$, $ u\in {\cal F}^{-1}{\cal E}'$ clearly give
\begin{equation}
  |\varphi*u(x)|\le \int (1+R|y|)^N|\varphi(y)|\frac{|u(x-y)|}{(1+R|y|)^N}\,dy
  \le c u^*(N,R;x).
  \label{conv-est}
\end{equation}
\end{example}

\begin{example}  \label{|u|*f-exmp}
Conversely $u^*(x)$ may be controlled by convolving $|u|$ with the
above function $f_N$; cf.\ \eqref{fuN-id}.
Hereby cases with $N>n$ are particularly simple as one has 
\begin{equation}
 u^*(N,R;x)\le C_N R^n f_N(R\cdot)*|u|(x).
  \label{u*-ineq}
\end{equation}
Indeed, when $u\in L_p$, $1\le p\le\infty$ with $\supp{\cal F}u\subset \Bbar(0,R)$
the compact spectrum of $u$ can be exploited by taking
$\psi$ as after \eqref{urp-ineq} 
above, which gives $u=R^n\psi(R\cdot )*u$.
Thence
\begin{equation}
  \begin{split}
  |u(y)|(1+R|x-y|)^{-N} &\le 
  \int \frac{R^n|\psi(R(y-z))u(z)|}{(1+R|x-y|)^N}\,dz
\\
  &\le 
  \int (1+R|y-z|)^N|\psi(R(y-z))|\frac{R^n|u(z)|}{(1+R|x-z|)^N}\,dz  
  \end{split}
  \label{uy-ineq}
\end{equation}
by using
$(1+R|x-y|)(1+R|y-z|)\ge(1+R|x-z|)$ in the denominator. This gives
\begin{equation}
  u^*(x)\le C_{N}\int \frac{R^n |u(z)|}{(1+R|x-z|)^N}\,dz,
  \label{CN-ineq}
\end{equation}
where $C_{N}:=\sup (1+|v|)^N|\psi(v)|<\infty$ because $\psi\in {\cal S}$.
This shows the claim in \eqref{u*-ineq}.
\end{example}

As an addendum to Example~\ref{|u|*f-exmp}, a basic estimate gives 
in \eqref{u*-ineq} that for $p\ge 1$
\begin{equation}
  \nrm{u^*}{p} \le C_N \nrm{R^nf_N(R\cdot)*|u|}{p}
 \le C_N \int (1+|z|)^{-N}\,dz\cdot \nrm{u}{p}.
  \label{u*N-est}
\end{equation}
So for $N>n$ this short remark proves a special case of 
the maximal inequality \eqref{max-ineq'}.

However, $N>n$ is far from an optimal assumption for \eqref{max-ineq'}. 
But a few changes give the improvement $N>n/p$;
and also every $p\in \,]0,\infty]$ can be treated using the
Nikolski\u\i--Plancherel--Polya inequality \eqref{urp-ineq}. 

The idea is to utilise the powerful pointwise estimate in
\eqref{u*-ineq},
where e.g.\ both sides can be integrated over $\Rn$
(unlike \eqref{uu*-ineq}). But first it is generalised thus:
 
\begin{proposition}   \label{u*p-prop}
If $u\in {\cal S}'$, $\supp{\cal F}u\subset\Bbar(0,R)$ 
and $N,p\in \,]0,\infty [\,$ are arbitrary, then
\begin{equation}
 u^*(N,R;x)\le C_{n,N,p} 
  \big(\int R^n \frac{|u(x-z)|^p}{(1+R|z|)^{Np}}\,dz\big)^{\fracpi}
  =C_{n,N,p}\big( R^n f_N^p(R\cdot)*|u|^p(x)\big)^{\fracpi}
  \label{u*p-ineq}
\end{equation}
for a constant $C_{n,N,p}$ depending only on $n$, $N$ and $p$.
\end{proposition}

\begin{proof}
As above $u(x)=R^n\psi(R\cdot )*u=\dual{u}{R^n\psi(R(x-\cdot ))}$, which can be
written as an integral since $(1+|y|)^{-k}u(y)$ is in $L_1(\Rn)$ for a large
$k$; therefore \eqref{uy-ineq} holds.
Suppose now that the right-hand side of \eqref{u*p-ineq}
is finite.

For $1\le p<\infty $ one can simply use H{\"o}lder's inequality for
$p+p'=p'p$ in the passage from \eqref{uy-ineq} to \eqref{CN-ineq}; then
$C_{n,N,p}=(\int (1+|z|)^{Np'}|\psi(z)|^{p'}\,dz)^{1/p'}$ gives
\eqref{u*p-ineq}.

If $0<p\le 1$ the $L_1$-norm with respect to $z$ in \eqref{uy-ineq}
can be estimated by the $L_p$-norm, according to 
\eqref{urp-ineq}, for the Fourier transform
of $z\mapsto \psi(R(y-z))u(z)$ is supported by $\Bbar(0,3R)$.
Invoking the specific constant in \eqref{urp-ineq} and proceeding as before,
this gives 
\begin{equation}
  \begin{split}
  u^*(x)^p&\le \int\sup_{y\in \Rn}
 \frac{\nrm{\psi}{\infty }^{(\tfrac{1}{p}-1)p}(3R)^{(\tfrac{n}{p}-n)p}C_N^p 
       R^{np}|u(z)|^p}{(1+R|x-y|)^{Np}(1+R|y-z|)^{Np}}\,dz
\\
  &\le C_{n,N,p}^p \int\frac{R^{n}|u(x-z)|^p}{(1+R|z|)^{Np}}\,dz,
  \end{split}
\end{equation}
where $C_N$ is as in \eqref{CN-ineq} and now
$C_{n,N,p}=C_N 3^{n/p}\nrm{\psi}{\infty }^{(1-p)/p}$.
\end{proof}

These elementary considerations give a short proof, in the style of
\eqref{u*N-est}, of the following 
important theorem on the $L_p$-boundedness of the maximal operator
$u\mapsto u^*$.

\begin{theorem}
  \label{max-thm}
When $0<p\le\infty$ and $N>n/p$, then there is a constant
$C'_{n,N,p}>0$ such that the maximal function 
$u^*(N,R;x)$ in \eqref{u*-eq} fulfils
\begin{equation}
  \nrm{u^*(N,R;\cdot)}{p}\le C'_{n,N,p}\nrm{u}{p}
  \label{max-ineq}
\end{equation}
for every $u\in L_p(\Rn)\cap{\cal S}'(\Rn)$ in the closed subspace with 
$\supp{\cal F}u\subset \Bbar(0,R)$. On this subspace there is Lipschitz
continuity
\begin{equation}
  \nrm{u^*(N,R;\cdot )-v^*(N,R;\cdot )}{p}\le C'_{n,N,p} \nrm{u-v}{p}.
\end{equation}
\end{theorem}

\begin{proof}
Lemma~\ref{uu*-lem} yields that $u^*$ is finite 
and consequently continuous as noted after \eqref{u*xy-ineq}, 
hence measurable.
The case $p=\infty $ then follows at once from the lemma.
For $0<p<\infty $ one can integrate both sides of
\eqref{u*p-ineq}, which by Fubini's theorem
yields
\begin{equation}
  \int |u^*(x)|^{p}\,dx  \le C_{n,N,p}^p
   \iint\frac{|u(x-z/R)|^{p}}{(1+|z|)^{Np}} \,dz\,dx
   =C^p\int |u(x)|^{p}\,dx
  \end{equation}
for $C^p=C_{n,N,p}^p\int(1+|z|)^{-Np}\,dz$. Since $Np>n$ 
this gives $u^*\in L_p$ and \eqref{max-ineq}.
Now the Lipschitz property follows by integration on both sides of
\eqref{uvmu-eq}.
\end{proof}

Among the further properties there is a Bernstein inequality for $u^*$,
which states that the maximal function of $u$
controls that of the derivatives $\partial^\alpha u$.

\begin{proposition}  \label{Bern-prop}
The estimate $(\partial^\alpha u)^*(N,R;x)\le C^{(\alpha)}_N R^{|\alpha|} u^*(N,R;x)$
is valid when
$\supp{\cal F}u\subset \Bbar(0,R)$, with a constant $C^{(\alpha)}_N$ independent of $u$, $R$.
\end{proposition}
While this 
is known (cf.\ \cite[1.3.1]{T2} for $R=1$), it is natural
to give the short proof here.
Writing $u(x-y)(1+R|y|)^{-N}$ in terms of the convolution $R^n\psi(R\cdot
)*u$, cf.\ Example~\ref{|u|*f-exmp},
it is straightforward to see by applying $\partial^\alpha_x$ that for $N'>0$,
\begin{equation}
  (\partial^\alpha u)^*(N,R;x)\le \sup(1+|\cdot |)^{N'}|\partial^\alpha\psi|
  \sup_y \int \frac{R^{n+|\alpha|} |u(z)|}{(1+R|y|)^N(1+R|x-y-z|)^{N'}}\,dz.
\end{equation}
For $N'=N+n+1$ a simple estimate of the denominator, 
cf.\ Example~\ref{|u|*f-exmp}, now shows Proposition~\ref{Bern-prop}
with the constant 
$C^{(\alpha)}_N=\sup(1+|\cdot |)^{N+n+1}|\partial^\alpha\psi|\int(1+|z|)^{-n-1}\,dz$.

\begin{remark}   \label{max-rem}
The maximal function $u^*$ was introduced by Peetre \cite{Pee75TL},
inspired by the non-tangential maximal function used by Fefferman and
Stein a few years earlier \cite{FeSt72}. 
It has been widely used in the theory of Besov
and Lizorkin--Triebel spaces, cf.\ \cite{T2,T3,RuSi96}, where
the boundedness in Theorem~\ref{max-thm} has been a main tool since the
1970's; cf.\ \cite[1.4.1]{T2}.
Usually its proof has been based on an estimate in terms of 
the Hardy--Littlewood maximal function,
$M_ru(x)=\sup_{\rho} (\rho^{-n}\int_{|y|<\rho} |u(x+y)|^r\,dy)^{1/r}$, 
i.e.\ for $\supp{\cal F} u\subset \Bbar(0,R)$,
\begin{equation}
 u^*(N,R;x)\le c M_r u(x),\qquad N\ge n/r. 
  \label{u*Mu-eq}
\end{equation}
When $N>n/r$ this results from Proposition~\ref{u*p-prop} by splitting the
integral ($p=r$) in regions with $2^k\le |z|\le 2^{k+1}$.
(For $N=n/r$ it was shown by Triebel, cf.\ \cite[1.3.1 ff]{T2}, by combining an inequality
for $u^*$, $(\partial_j u)^*$ and $M_r u$, due to Peetre \cite{Pee75TL}, with
the Bernstein inequality for $u^*$; cf.\ Proposition~\ref{Bern-prop}.)
This gave a proof of \eqref{max-ineq'} by combining
\eqref{u*Mu-eq} with the inequality 
$\nrm{M_ru}{p}\le c\nrm{u}{p}$ for $p>r$. 
The present proofs of Proposition~\ref{u*p-prop} and 
Theorem~\ref{max-thm} are rather simpler.
\end{remark}


\section{Preparations}    \label{prep-sect}

Notation and notions from distribution theory are the same as in
H{\"o}rmander's book~\cite{H}, unless otherwise mentioned.
E.g.\ $[t]$ denotes the largest integer $k\le t$ for $t\in\R$.
The Fourier transformation is ${\cal F}u(\xi)=\int e^{-\im x\cdot \xi}u(x)\,dx$,
which will be written as ${\cal F}_{x\to\xi}u(x,y)$ when $u$ depends on further
variables $y$. The value of $u\in\cal S'(\Rn)$ on the Schwartz function
$\psi\in\cal S(\Rn)$ is denoted by $\dual{u}{\psi}$.

\bigskip

As mentioned in the introduction the paper deals with operators given by 
\begin{equation}
  a(x,D)u=\OP(a)u(x)
  =(2\pi)^{-n}\int e^{\im x\cdot \eta} a(x,\eta){\cal F}u(\eta)\,d\eta,
 \qquad u\in {\cal S}(\Rn).
  \label{axDu-eq}
\end{equation}
Hereby the symbol $a(x,\eta)$ is $C^\infty $ on $\Rn\times \Rn$ and is
taken to fulfil the H{\"o}rmander condition of order $d\in \R$, i.e.\ for all
multiindices $\alpha$, $\beta\in \N_0^n$ there exists a constant
$C_{\alpha,\beta}>0$ such that 
\begin{equation}
  |D^\alpha_\eta D^\beta_x a(x,\eta)|\le C_{\alpha,\beta}
  (1+|\eta|)^{d-\rho|\alpha|+\delta|\beta|}.
  \label{Srd-eq}
\end{equation}
The space of such symbols is denoted by $S^d_{\rho,\delta}(\Rn\times \Rn)$
or $S^d_{\rho,\delta}$;
and $S^{-\infty }:=\bigcap S^d_{\rho,\delta}$.

The parameters $\rho$, $\delta\in [0,1]$ 
are mainly assumed to fulfil $\delta<\rho$, so that $a(x,D)$ by duality has a
continuous extension $a(x,D)\colon {\cal S}'(\Rn)\to{\cal S}'(\Rn)$.
(Type $1,1$-operators, i.e.\ $\delta=1=\rho$, are considered briefly in
Section~\ref{t11-sect} below.)
If desired the reader may specialise to the classical case
$\rho=1$, $\delta=0$.

Together with $a(x,D)$ one has the distribution kernel 
$K(x,y)={\cal F}^{-1}_{\eta\to z}a(x,\eta)\big|_{z=x-y}$ that in the usual way is seen to
be $C^\infty $ for $x\ne y$ (also for $a\in S^d_{1,1}$). It fulfils 
\begin{equation}
  \dual{a(x,D)u}{\varphi}=\dual{K}{\varphi\otimes u}
  \quad\text{for all}\quad
  u,\varphi\in {\cal S}.
\end{equation}

As preparations, two special cases are considered: 
if $u=v+v'$ is any splitting of $u\in{\cal S}+{\cal F}^{-1}{\cal E}'$ 
with $v\in {\cal S}$ and $v'\in {\cal F}^{-1}{\cal E}'$ then
\begin{equation}
  a(x,D)u= a(x,D)v+\OP(a(1\otimes \chi))v',
  \label{aFE-eq}
\end{equation}
whereby $a(1\otimes \chi)(x,\eta)=a(x,\eta)\chi(\eta)$ and $\chi\in
C^\infty_0(\Rn)$ is chosen so that $\chi=1$ holds in a neighbourhood of
$\supp{\cal F}v'$,
or just on
a neighbourhood of the smaller set
\begin{equation}
  \bigcup_{x\in \Rn}\supp a(x,\cdot ){\cal F} v'(\cdot ).
  \label{aFE-eq'}
\end{equation}
Indeed, by linearity on the left-hand side of \eqref{aFE-eq} the identity results,
for the term $a(x,D)v'$ equals $\OP(a(1\otimes \chi))v'$ if 
$v'\in {\cal F}^{-1}C^\infty_0(\Rn)$ that extends to $v'\in {\cal F}^{-1}{\cal E}'$ by
mollification of ${\cal F}v'$ since $a(1\otimes\chi)\in S^{-\infty }$.

Moreover, for every auxiliary function  $\psi\in C^\infty_0(\Rn)$ equal to
$1$ in a neighbourhood of the origin, continuity of the adjoint operation
$a\mapsto e^{\im D_x\cdot D_\eta}\bar a$ yields 
\begin{equation}
  a(x,D)u=
  \lim_{m\to\infty }\op{OP}(\psi(2^{-m}D_x)a(x,\eta)\psi(2^{-m}\eta))u.
  \label{aPsi-eq}
\end{equation}

\section{Pointwise estimates}   \label{pest-sect}
This section develops a flexible framework for discussion of
pseudo-differential operators. These are only for convenience restricted to
the classes recalled in Section~\ref{prep-sect}.

\subsection{The factorisation inequality}   \label{fact-ssect}
The simple result below introduces $u^*(x)$ as a fundamental tool
for the proof of \eqref{fact-ineq}, hence of \eqref{au*-eq}. 
It is therefore given as a theorem.

Formally the idea is to proceed as in Example~\ref{conv-exmp}, 
cf.\ \eqref{conv-est}, now departing from 
\begin{equation}
  a(x,D)u(x)=\int K(x,x-y)u(x-y)\,dy.
\end{equation}
This leads to the \emph{factorisation inequality} \eqref{axu*x-eq} below, where the
dependence on $a(x,\eta)$ is taken out in the symbol factor $F_a$, also
called the ``$a$-factor''. This is
essentially a weighted $L_1$-norm of the
distribution kernel. (The estimate shows that the case of an operator is not
much worse than that of $\varphi*u$ in Example~\ref{conv-exmp}.)

\begin{theorem}   \label{Fau*-thm}
Let $a\in S^d_{\rho,\delta}(\Rn\times\Rn)$ for $0\le \delta< \rho\le 1$. 
When $u\in {\cal S}'(\Rn)$ 
with $\supp\hat u\subset\Bbar(0,R)$, then one has 
the following pointwise estimate for all $x\in \Rn$:
\begin{equation}
  |a(x,D)u(x)|\le F_{a}(N,R;x)  \cdot u^*(N,R;x).
  \label{axu*x-eq}
\end{equation}
Hereby $u^*$ is as in \eqref{u*-eq} while $F_{a}$ 
is bounded and continuous for $x\in \Rn$ and is given in terms of
an auxiliary function $\chi\in C^\infty_0(\Rn)$ equal to
$1$ on a neighbourhood of $\supp\hat u$ as
\begin{equation}
  F_{a}(N,R;x)=\int_{\Rn}(1+R|y|)^N|
  {\cal F}^{-1}(a(x,\cdot)\chi(\cdot))|\,dy.
  \label{Fa-id}
\end{equation}
The inequality \eqref{axu*x-eq} holds for $N>0$, 
and remains true if $\chi=1$ on 
$\bigcup_{x\in \Rn} \supp a(x,\cdot )\hat u(\cdot) $.
\end{theorem}
\begin{proof}
Using formula \eqref{aFE-eq} with $v'=u$, and \eqref{aFE-eq'} for the last statement, 
\begin{equation}
  a(x,D)u(x)=\OP(a(1\otimes \chi))u
  =\dual{u}{{\cal F}_{\eta\to y}
        (\tfrac{e^{\im x\cdot \eta}}{(2\pi)^n}a(x,\eta)\chi(\eta))}
  \label{uFachi-eq}
\end{equation}
for the last rewriting is evident from \eqref{axDu-eq} if 
$u\in {\cal F}^{-1}C^\infty_0$ 
and follows for general $u\in {\cal F}^{-1}{\cal E}'$ by mollification of 
${\cal F}u$, since $a(1\otimes \chi)$ is in $S^{-\infty }$.

Now $a(x,\eta)\chi(\eta)$ is in $C^\infty_0(\Rn)$ for fixed $x$, so
$y\mapsto {\cal F}^{-1}_{\eta\to y}(a(1\otimes\chi))(x,x-y)$ 
decays rapidly while $u(y)$ grows
polynomially by the Paley--Wiener--Schwartz theorem. Therefore the above 
scalar product on ${\cal S}'\times{\cal S}$ is an integral,
so by the change of variables $y\mapsto x-y$,
\begin{equation}
  \begin{split}
    |a(x,D)u(x)|&=|\int u(x-y)
  {\cal F}^{-1}_{\eta\to y}(a(1\otimes\chi))(x,y)\,dy|
\\
   &\le \sup_{z\in \Rn}\frac{|u(x-z)|}{(1+R|z|)^N}
   \int(1+R|y|)^N|{\cal F}^{-1}_{\eta\to y}(a(1\otimes\chi))(x,y)|\,dy
\\
  &=u^*(x)F_a(x),
  \end{split}
\label{auFa-ineq}
\end{equation}
according to the definition of $u^*(x)$ in \eqref{u*-eq} and that of
$F_a(x)$ in \eqref{Fa-id}.
 
That $x\mapsto F_a(x)$ is bounded follows by insertion of
$1=(1+|y|^2)^{N'}(1+|y|^2)^{-N'}$ for $N'>(N+n)/2$ 
since ${\cal F}^{-1}_{\eta\to y}((1\mlap_{\eta})^{N'}[a(x,\eta)\chi(\eta)])$
is bounded with respect to $(x,y)$
because of the compact suppport of $\chi$. 
These estimates also yield continuity of the symbol factor $F_a(x)$.
\end{proof}

Disregarding the spectral radius $R$ and $N$, \eqref{axu*x-eq}
may be written concisely as
\begin{equation}
  |a(x,D)u(x)|\le F_a(x)\cdot u^* (x).
  \label{fact'-ineq}
\end{equation}
It is noteworthy that the entire
influence of the symbol lies in the $a$-factor $F_a(x)$, while
$u$ itself is mainly felt in $u^*(x)$. It is only in a vague way, 
i.e.\ through $N$ and $R$, that $u$ contributes to $F_a(x)$,
so the factorisation inequality is rather convenient.

The theorem is also valid more generally; 
e.g.\ Section~\ref{t11-sect} gives an extension to symbols of type $1,1$
(extensions to other general symbols can undoubtedly be worked out when
needed). 
To give a version for functions
without compact spectrum, ${\cal O}_M(\Rn)$ will as usual stand for the space
of slowly increasing functions, i.e.\ the $f\in C^\infty (\Rn)$ satisfying the
estimates
\begin{equation}
  |D^\alpha f(x)|\le C_\alpha(1+|x|)^{N_\alpha} .
  \label{OM-eq}
\end{equation}
Analogously to the argument after \eqref{FuBR-eq},  
$f^*(N,R;\cdot)$ is finite for $N\ge N_{(0,\dots ,0)}$, any $R>0$.
There is a factorisation inequality for such functions, 
at the expense of a sum over its derivatives:

\begin{theorem}
  \label{FaOM-thm}
When $a(x,\eta)$ is in $S^d_{\rho,\delta}(\Rn\times \Rn)$, $1\le \delta<\rho\le 1$, and
$u\in {\cal O}_M(\Rn)$ while $N'>(d+n)/2$ is a non-negative integer, 
then one has for $N$, $R>0$ that
\begin{equation}
  |a(x,D)u(x)|\le c F_a(N,R;x) 
  \sum_{|\alpha|\le 2N'} (D^\alpha u)^*(N,R;x),
  \label{fact''-ineq}
\end{equation}
where $F_a$ is defined  by \eqref{Fa-id} for $\chi(\eta)=(1+|\eta|^2)^{-N'}$
and again is in $C(\Rn)\cap L_\infty (\Rn)$.
\end{theorem}
\begin{proof}
That $F_a$ is in $C(\Rn)\cap L_\infty (\Rn)$ can be seen as
above, for $a(x,\eta)\chi(\eta)\in S^{d-2N'}_{\rho,\delta}$ is
integrable with respect to $\eta$.
When $a\in S^{-\infty }$ and $u\in {\cal S}$,
\begin{equation}
  a(x,D)u(x)= 
  \int (1\mlap)^{N'} u(y){\cal F}_{\eta\to y}(\tfrac{e^{\im x\cdot
\eta}}{(2\pi)^n} a(x,\eta)\chi(\eta))\,dy.
  \label{aOM-eq}
\end{equation}
By continuity this extends to all $u\in {\cal S}'$, in particular to $u\in
{\cal O}_M$;
and since $S^{-\infty}$ is dense in $S^{d'}_{\rho,\delta}$ for $d'>d$, it extends 
then to all $a\in S^d_{\rho,\delta}$ since
$(1+|y|)^{-N''}(1\mlap)^{N'} u(y)$ is in $L_1$ for a large $N''$.
In the same way as in \eqref{auFa-ineq} this yields
\begin{equation}
  |a(x,D)u(x)|\le F_a(N,R;x) ((1\mlap)^{N'}u)^*(N,R;x).
\end{equation}
Since $(1\mlap)^{N'}u=\sum_{|\alpha|\le 2N'} c_{\alpha,N'}D^\alpha u$,
subadditivity of the maximal operator gives the rest.
\end{proof}

As a first consequence of the factorisation inequalities,
when $u\in {\cal O}_M$ then $a(x,D)u$ is of polynomial growth by 
\eqref{fact''-ineq}, and continuous by \eqref{aOM-eq}; indeed,
$F_a(x)$ is bounded and $D^\alpha u\in {\cal O}_M$
so $(D^\alpha u)^*(x)$ has such growth for
$N$ sufficiently large; cf.\ \eqref{u*xy-ineq}. 
Moreover, this applies to the commutator $[D^\beta,a(x,D)]$, say in
the class $\OP(S^{d+|\beta|}_{\rho,\delta})$, so also
$D^\beta a(x,D)u$ has polynomial growth. Which altogether proves

\begin{corollary}   \label{aOO-cor}
$a(x,D)$ is a map ${\cal O}_M(\Rn)\to {\cal O}_M(\Rn)$ when 
$a\in S^d_{\rho,\delta}$, $0\le \delta<\rho\le 1$.
\end{corollary}
While this is known for $\rho=1$, $\delta=0$ from e.g.\
\cite[Cor.~3.8]{SRay91}, the above version for 
the general case $0\le \delta<\rho\le 1$ is rather more direct.

Secondly, one may now obtain the
$L_p$-estimate mentioned in the introduction.

\begin{corollary}
  \label{aFE-cor}
For each $a\in S^d_{\rho,\delta}(\Rn\times \Rn)$, $0\le \delta<\rho\le 1$, 
and $p\in \,]0,\infty ]$ there is
to every $R\ge 1$, $N>n/p$ a constant $C(N,R)$ such that 
\begin{equation}
  \nrm{a(x,D)u}{p}\le C(N,R) \nrm{u}{p}
  \label{iLp-eq'}
\end{equation}
whenever $u\in L_p(\Rn)\bigcap{\cal S}'(\Rn)$,
fulfils
$\supp\hat u\subset \overline{B}(0,R)$.
\end{corollary}
\begin{proof} By taking $L_p$-norms on both sides of the factorisation
inequality, \eqref{iLp-eq'} results with $C(N,R)=C'_{n,N,p}\sup_x|F_a(N;R;x)|$,
cf. Theorem~\ref{max-thm}; 
this is finite according to Theorem~\ref{Fau*-thm}.
\end{proof}

Since the spectral condition on $u$ implies $u\in C^\infty $, it is
hardly surprising that the above $L_p$-result is valid for arbitrary orders
$d\in \R$. In fact it may, say for $1<p<\infty $, $(\rho,\delta)=(1,0)$, 
be proved simply by observing that $a(x,D)$ has the same
action on $u$ as some $b(x,D)\in \OP(S^{-\infty })$
so that boundedness of $b(x,D)$ on $L_p$ gives the rest.

It is noteworthy, however, that the existing proofs of
$L_p$-boundedness use fundamental parts of real analysis,
e.g.\ Marcinkiewicz interpolation and the Calderon--Zygmund lemma.
In contrast to this,  pointwise estimates lead straightforwardly to 
Corollary~\ref{aFE-cor}. This evident efficacy is also clear from 
the easy extension to the full range $0<p\le \infty $ and to type
$1,1$-operators in Section~\ref{t11-sect}.

\subsection{Estimates of the symbol factor}   \label{Faest-ssect}

To utilise Theorem~\ref{Fau*-thm} it is of course vital to control 
$F_a$. This
leads directly to integral conditions on $a$, similarly 
to the Mihlin--H{\"o}rmander theorem. 

\begin{theorem}
  \label{Fa-thm}
Assume $a(x,\eta)$ is in $S^d_{\rho,\delta}(\Rn\times\Rn)$, 
$0\le \delta<\rho\le 1$, and let $F_a(N,R;x)$ be given by   
\eqref{Fa-id} for parameters $R,N>0$, whereby the auxiliary function is
taken as $\chi=\psi(R^{-1}\cdot)$ for $\psi\in C^\infty_0(\Rn)$ equalling
$1$ in (the closure of) an open set. Then 
\begin{equation}
 0\le F_a(x) \le c_{n,k} \sum_{|\alpha|\le k} 
  (\int_{R\supp \psi} |R^{|\alpha|}D^{\alpha}_\eta a(x,\eta)|^2
    \,\frac{d\eta}{R^n})^{1/2}
  \label{FaMH-eq}
\end{equation}
for all $x\in \Rn$,
when $k$ is the least integer satisfying $k>N+n/2$.
\end{theorem}

First it is convenient to recall that, 
for $z\in \Rn$ and $k\in \N$, an expansion yields 
\begin{equation}
  (1+|z|)^k\le \sum_{j=0}^k \binom kj (|z_1|+\dots  +|z_n|)^j
  =\sum_{|\alpha|\le k}C_{k,\alpha} |z^\alpha|.
  \label{zN-ineq}
\end{equation}

\begin{proof} 
The idea is to pass to the $L_2$-norm in \eqref{Fa-id}
using Cauchy--Schwarz' inequality and that 
$(\int R^n(1+|Ry|)^{-n-\varepsilon}\,dy)^{1/2}<\infty$
for $\varepsilon>0$.
Thus, if $\varepsilon$ is so small that $k\ge  N+(n+\varepsilon)/2$,
\begin{equation}
  F_a(N,R;x)\le c_nR^{-n/2}(\int (1+|Ry|)^{2k}
        |{\cal F}^{-1}_{\eta\to y}[a(x,\cdot )\psi(R^{-1}\cdot )]|^2\,dy)^{1/2}.
\end{equation}
Applying \eqref{zN-ineq} to $z=Ry$ 
and `commuting' the resulting polynomials $(Ry)^\alpha$ with 
the inverse Fourier transformation, it is seen that for fixed $x\in \Rn$,
\begin{equation}
  \begin{split}
    F_a(x)&\le c_n R^{-n/2} \sum_{|\alpha|\le k} C_{k,\alpha}
   (\int|{\cal F}^{-1}_{\eta\to y}
   [(\im R\partial_\eta)^{\alpha}
    a(1\otimes \psi(R^{-1}\cdot ))](x,y)|^2\,dy)^{1/2}
\\
  &\le c\sum_{|\alpha+\beta|\le k}
   (\int_{R\supp \psi} (R^{|\alpha+\beta|}|D^{\alpha}_\eta a(x,\eta)|
   |D^{\beta}(\psi(R^{-1}\eta))|)^2\,\frac{d\eta}{R^n})^{1/2}.
  \end{split}
  \label{Fapsi-ineq}
\end{equation}
Since $D^{\beta}(\psi(R^{-1}\cdot ))=R^{-|\beta|}(D^\beta\psi)(R^{-1}\cdot
)$ is bounded, the result follows.
\end{proof}

\begin{remark}
As an alternative to the estimate $|a(x,D)u(x)|\le  F_a(x)u^*(x)$, it
deserves to be mentioned that other useful properties can be obtained in a
similar fashion:
by defining an $a$-factor in terms of an $L_2$-norm, i.e.\ 
\begin{equation}
  \tilde F_a(N,R;x)^2=\int_{\Rn}(1+|Ry|)^{2N}|{\cal F}^{-1}_{\eta\to y}
  (a(x,\cdot)\chi(\cdot))|^2\,dy,
  \label{tFa-id}
\end{equation}
the Cauchy--Schwarz inequality gives
\begin{equation}
  |a(x,D)u(x)|\le \tilde F_a(N,R;x)
    (\int_{\Rn}\frac{|u(x-y)|^2}{(1+|Ry|)^{2N}}\,dy)^{1/2}
  \le c\tilde F_a(N,R;x)u^*(\varepsilon,R;x)
\end{equation}
where $c=(\int (1+|Ry|)^{-2(N-\varepsilon)}\,dy)^{1/2}<\infty $
whenever $N>n/2+\varepsilon$ for some $\varepsilon>0$.

For one thing $\tilde F_a^2\in C^\infty (\Rn)$, with bounded
derivatives of any order.
Secondly, this gives a version of Theorem~\ref{Fa-thm}
where only estimates with $|\alpha|\le[n/2]+1$ 
is required, as in the Mihlin--H{\"o}rmander theorem.
But it would not be feasible in general to replace $u^*(N,R;x)$ by
$u^*(\varepsilon,R;x)$ for small $\varepsilon$ as above, so 
$\tilde F_a(x)$ is only mentioned in this remark.
\end{remark}

Although it is a well-known exercise to control
\eqref{FaMH-eq} in terms of symbol seminorms, it is
important to control the behaviour with respect to $R$ and to verify that it
improves when $a(x,\cdot )\hat u(\cdot )$ is supported in a corona. 
Therefore the special case in
\eqref{cpsi-eq} below is included: 

\begin{corollary}
  \label{Fa-cor} 
Assume $a\in S^{d}_{1,\delta}(\Rn\times\Rn)$, $0\le \delta<1$, 
and let $N$, $R$ and $\psi$ have the same meaning as in Theorem~\ref{Fa-thm}.
When $R\ge 1$ and $k>N+n/2$, $k\in \N$,
then there is a seminorm $p$ on $S^d_{1,\delta}$ and some 
$c_k>0$ independent of $R$ such that 
\begin{equation}
 0\le F_a(x) \le c_k p(a) R^{\max(d,k)} \quad\text{for all}\quad x\in
\Rn. 
  \label{Rmax-eq}
\end{equation}
Moreover, if $\supp\psi$ is contained in a corona
\begin{equation}
  \{\,\eta\mid \theta_0 \le|\eta|\le \Theta_0 \,\},
  \label{cpsi-eq}
\end{equation}
and 
$\psi(\eta)=1$ holds for $\theta_1\le |\eta|\le \Theta_1$, whereby
$0\ne\theta_0<\theta_1<\Theta_1<\Theta_0$,
then
\begin{equation}
  0\le F_a(x)\le c'_kR^{d}p(a) \quad\text{for all}\quad x\in \Rn,
  \label{Rd-eq}
\end{equation}
with $c'_k=c_k\max(1,\theta_0^{d-k},\theta_0^d)$.
\end{corollary}
\begin{remark} \label{asymp1-rem}
For general $\rho\in \,]0,1]$, the asymptotics of $F_a(x)$ for $R\to\infty $ 
corresponding to \eqref{Rmax-eq}, \eqref{Rd-eq} will be 
${\cal O}(R^{\max(d+(1-\rho)k,k)})$
and ${\cal O}(R^{d+(1-\rho)k})$, respectively.
Details are left out for simplicity's sake.
\end{remark}

\begin{proof} 
Setting
$p_{\alpha,\beta}(a)=\sup
(1+|\eta|)^{-d+|\alpha|-\delta|\beta|}|D^\beta_xD^\alpha_\eta a(x,\eta)|$
and continuing from the proof of Theorem~\ref{Fa-thm}, 
the change of variables $\eta=R\zeta$ gives
\begin{equation}
  \begin{split}
  F_a(x) & \le c\sum_{|\alpha|\le k}     p_{\alpha,0}(a)
   (\int_{\supp\psi} |(1+|R\zeta|)^{d-|\alpha|} R^{|\alpha|}|^2
   \,d\zeta)^{\tfrac{1}{2}} 
\\
  &\le C'_{n,k} R^{\max(d,k)}\sum_{|\alpha|\le k}p_{\alpha,0}(a).
   \end{split}
  \label{Fap-ineq}
\end{equation}
In fact $d\ge k\ge |\alpha|$ gives
$R^{|\alpha|}(1+R|\zeta|)^{d-|\alpha|}\le R^d(1+|\zeta|)^{d-|\alpha|}$ 
since $R\ge1$; 
for $d<k$ the crude estimate 
$R^{|\alpha|}(1+R|\zeta|)^{d-|\alpha|}\le R^{k}$ applies e.g.\ for $|\alpha|=k$.
This shows \eqref{Rmax-eq}.

In case $\psi$ is supported in a corona as described,
$d-|\alpha|<0$ and $\zeta\in \supp\psi$ entail
\begin{equation}
  (1+|R\zeta|)^{d-|\alpha|}R^{|\alpha|}
  \le (R\theta_0)^{d-|\alpha|} R^{|\alpha|}
  \le \max(\theta_0^{d-k},\theta_0^d)R^d.
\end{equation}
This yields an improvement of \eqref{Fap-ineq} for terms with
$|\alpha|>d$; thence \eqref{Rd-eq}. 
\end{proof}

As desired Corollary~\ref{Fa-cor} shows that the $a$-factor $F_a(x)$ 
has its sup-norm bounded by a symbol seminorm. 
This applies of course in $|a(x,D)u(x)|\le F_a(x)u^*(x)$.

In this connection, one could simply take $R$ equal to the spectral
radius of $u$, or if possible $R$ so large that the corona 
$\{\,\eta\mid \theta_1R\le |\eta|\le \Theta_1R\,\}$
is a neighbourhood of $\supp a(x,\cdot )\hat u(\cdot )$ for all $x$;
cf~\eqref{Rmax-eq} and \eqref{Rd-eq}.

However, a good choice of $N$ is a more delicate question, which in general involves the
order of ${\cal F}u$ as a distribution.
E.g.\ $N\ge\order({\cal F} u)$ was seen in Section~\ref{PFS-sect} to imply that
$u^*(N,R;x)$ is finite everywhere. This was relaxed completely to $N>0$
for $u\in L_p\cap {\cal F}^{-1}{\cal E}'$ in Lemma~\ref{uu*-lem}; moreover,
for arbitrary $u\in L_p$ with $1\le p\le 2$, the order of ${\cal F}u$ is $0$,
so $u^*$ is finite regardless of $N>0$.

Especially for functions $u$ in Sobolev spaces $H^s$ the function $u^*$ is
always finite for $N>0$. Therefore it is harmless that the estimates 
in Corollary~\ref{Fa-cor} depend on $N$, for only
seminorms $p_{\alpha,0}(a)$ with $|\alpha|\le 1+[n/2+N]$ enters
there, and by taking $0<N< 1/2$ in both odd and even dimensions
estimates of $a(x,D)u(x)$ with $u\in \bigcup H^s$ only requires
the well-known estimates of $D^\alpha_{\eta}a(x,\eta)$ for 
$|\alpha|\le [n/2]+1$. 

However, in connection with $L_p$-bounds of $u^*(x)$, one is often forced to take
$N>n/p$ in the $L_p$-estimates of $a(x,D)u$; cf.\ Theorem~\ref{max-thm}.

\bigskip

In addition to high frequencies removed by the spectral cut-off function $\chi$ in
Theorem~\ref{Fau*-thm}, the symbols dependence on $x$
may be frequency modulated by means of a Fourier multiplier
$\varphi(Q^{-1}D_x)$, which depends on a second spectral quantity $Q>0$.
For the modified symbol
\begin{equation}
  a_{Q}(x,\eta)=\varphi(Q^{-1}D_x)a(x,\eta)
\end{equation}
and the corresponding symbol factor one can as shown below find its
asymptotics for $Q\to\infty $. In Littlewood--Paley theory, this is a
frequently asked question for $F_{a_Q}$:

\begin{corollary}
  \label{Fa0-cor}
When $a\in S^d_{1,\delta}(\Rn\times \Rn)$, $0\le\delta<1$ and $\varphi\in C^\infty_0(\Rn)$ with $\varphi=0$
in a neighbourhood of $\xi=0$, then there is a seminorm $p$ on $S^d_{1,\delta}$
and constants $c_M$, depending only on $M$, $n$, $N$, $\psi$ and $\varphi$,
such that for $R\ge 1$, $M>0$, $Q>0$,
\begin{equation}
  0\le F_{a_{Q}}(N,R;x)\le c_M p(a)Q^{-M} R^{\max(d+\delta M,[N+n/2]+1)}.
\label{Fa0-ineq}
\end{equation}
Here $d+\delta M$ can replace the maximum when the auxiliary function $\psi$ in
$F_{a_{Q}}$ fulfils the corona condition in Corollary~\ref{Fa-cor}.
\end{corollary}
\begin{proof}
Because
$a_{Q}(x,\eta)=\int Q^n\check \varphi(Qz)a(x-z,\eta)\,dz$, where $\check
\varphi$ has vanishing moments of every order, Taylor's formula with
remainder gives for any $M\in \N$
\begin{equation}
  a_{Q}(x,\eta)=\sum_{|\beta|=M}\tfrac{M}{\beta!}\int(-z)^\beta
Q^n\check\varphi(Qz) \int_0^1
(1-\tau)^{M-1}\partial^\beta_x a(x-\tau z,\eta)\,d\tau\,dz.
\end{equation}
Letting $z^{\beta}$ absorb $Q^{M}$ before substitution of $z$ by
$z/Q$, one finds
\begin{equation}
  \begin{split}
  Q^MF_{a_{Q}}(N,R;x)
&\le \sum_{|\beta|=M} \tfrac{M}{\beta!}
 \iiint  (1-\tau)^{M-1}(1+|z|)^{M}|\check \varphi(z)|(1+|Ry|)^N
\\
  &\hphantom{\le \sum_{|\beta|=M} \iiint } 
  \times  |{\cal F}^{-1}_{\eta\to y}
  (\partial^\beta_x a(x-\tau z/Q,\eta)\psi(\eta/R))|
  \,d\tau\,dz\,dy.  
  \end{split}
\end{equation}
Integrating first with respect to $y$ it follows by applying
Corollary~\ref{Fa-cor} to $\partial^{\beta}_x a \in
S^{d+\delta M}_{1,\delta}$ that,
by setting $p(a)=\sum p_{\alpha,\beta}(a)$
where $|\alpha|\le [N+n/2]+1$ and $|\beta|= M$, 
\begin{equation}
  F_{a_{Q}}(N,R;x)\le c_M p(a)R^{d+\delta M} Q^{-M}.
  \label{Faj-ineq}
\end{equation}
This is in case $\psi$ satisfies the corona condition. Otherwise the stated inequality
\eqref{Fa0-ineq} results. 
\end{proof}

\begin{remark}
In comparison with Remark~\ref{asymp1-rem}, the asymptotics for $R\to\infty$ are here
${\cal O}(R^{\max(d+\delta M+(1-\rho)k,k)})$ and ${\cal O}(R^{d+\delta M+(1-\rho)k})$, respectively,
for $k=[N+n/2]+1$.
\end{remark}

\begin{remark}
  \label{Marschall-rem}
As an alternative to the techniques in this section, Marschall's inequality 
gives a pointwise estimate for symbols $b(x,\eta)$ in 
$L_{1,\loc}(\R^{2n})\cap {\cal S}'(\R^{2n})$ 
with support in $\Rn\times \Bbar(0,2^k)$ and 
$\supp{\cal F} u\subset \Bbar(0,2^k)$, $k\in\N$:
\begin{equation}
  |b(x,D)v(x)|\le c\Nrm{b(x,2^k\cdot )}{\dot B^{n/t}_{1,t}} M_t u(x),
  \qquad 0<t\le 1.
  \label{Marschall-ineq}
\end{equation}
This goes back to \cite[p.37]{Mar85} and 
was exploited in e.g.\ \cite{Mar91,Mar95,Mar96}. In the above form it was proved in
\cite{JJ05DTL} under the natural condition that the right-hand side is in
$L_{1,\loc}(\Rn)$.
While $M_tu$ is as in Remark~\ref{max-rem}, the norm of the
homogeneous Besov space $\dot B^{n/t}_{1,t}$ on the symbol
is defined analogously to that of $B^{s}_{p,q}$ in \eqref{Bspq-id} below 
in terms of a partition of unity, though here with $1=\sum_{j=-\infty }^\infty
(\varphi(2^{-j}\eta)-\varphi(2^{1-j}\eta))$, $\eta\ne0$ so that \eqref{Bspq-id}
should be read with $\ell_q$ over $\mathbb{Z}$. This yields the
well-known dyadic scaling property that
\begin{equation}
  \Nrm{b(x,2^k\cdot )}{\dot B^{n/t}_{1,t}}
  = 2^{k(\frac nt-n)}\Nrm{b(x,\cdot)}{\dot B^{n/t}_{1,t}}.  
\end{equation}
While this can be useful, and indeed fits well into the framework of the
next section, cf.\ \cite{Mar91,Mar95,Mar96}, 
it is often simpler to use the factorisation inequality with $F_b$ and $u^*$ etc.
\end{remark}

\section{Littlewood--Paley analysis}
  \label{corona-sect}

In order to obtain $L_p$-estimates, it is
convenient to depart from the limit in \eqref{aPsi-eq}.
As usual the test function  $\psi$ there gives rise to a
Littlewood--Paley decomposition $1=\psi(\eta)+\sum_{j=1}^\infty 
\varphi(2^{-j}\eta)$ by setting $\varphi=\psi-\psi(2\cdot )$.
Note here that if $\psi\equiv 1$ for $|\eta|\le r$ while $\psi\equiv 0$ for
$|\eta|\ge R$, one can fix an integer $h\ge 2$ so that $2R< r2^h$.

Inserting twice into \eqref{aPsi-eq} that
$\psi(2^{-m}\eta)=\psi(\eta)+\varphi(2^{-1}\eta)+
\dots +\varphi(2^{-m}\eta)$, the so-called paradifferential splitting from the 1980's is
recovered: 
whenever $a(x,\eta)$ is in $S^d_{\rho,\delta}$, $0\le \delta<\rho\le 1$, and
$u\in {\cal S}'(\Rn)$,
\begin{equation}
  a_{\psi}(x,D)u=
  a_{\psi}^{(1)}(x,D)u+a_{\psi}^{(2)}(x,D)u+a_{\psi}^{(3)}(x,D)u,
  \label{a123-eq}
\end{equation}
whereby the expressions are given by 
the three series below (they converge in ${\cal S}'$),
\begin{align}
    a_{\psi}^{(1)}(x,D)u&=\sum_{k=h}^\infty \sum_{j\le k-h} a_j(x,D)u_k
  =\sum_{k=h}^\infty a^{k-h}(x,D)u_k
  \label{a1-eq}\\
  a_{\psi}^{(2)}(x,D)u&= \sum_{k=0}^\infty
               \bigl(a_{k-h+1}(x,D)u_k+\dots+a_{k-1}(x,D)u_k+a_{k}(x,D)u_k
\notag\\[-2\jot]
   &\qquad\qquad
                +a_{k}(x,D)u_{k-1} +\dots+a_k(x,D)u_{k-h+1}\bigr) 
  \label{a2-eq}\\
   a_{\psi}^{(3)}(x,D)u&=\sum_{j=h}^\infty\sum_{k\le j-h}a_j(x,D)u_k
   =\sum_{j=h}^\infty a_j(x,D)u^{j-h}.
  \label{a3-eq}
\end{align}
Here $u_k=\varphi(2^{-k}D)u$ while
$a_k(x,\eta)=\varphi(2^{-k}D_x)a(x,\eta)$;
by convention $\varphi$ is replaced by $\psi$ for $k=0$ and 
$u_k\equiv 0\equiv a_k$ for $k<0$.
In addition superscripts are used for the convenient shorthands $u^{k-h}$
and $a^{k-h}(x,D)$; e.g.\ the latter is given by
$a^{k-h}(x,D)=\sum_{j\le k-h}a_j(x,D)=\op{OP}(\psi(2^{h-k}D_x)a(x,\eta))$. 
Using this, there is a
brief version of \eqref{a2-eq},
\begin{equation}
 a_{\psi}^{(2)}(x,D)u=\sum_{k=0}^\infty
((a^{k}-a^{k-h})(x,D)u_k+a_k(x,D)(u^{k-1}-u^{k-h})). 
  \label{a2'-eq}
\end{equation}

The main point here is that the series in \eqref{a1-eq}--\eqref{a2'-eq} are easily treated
with the tools of the present paper. First of all, one has
the following inclusions for the spectra of the summands in 
\eqref{a1-eq}, \eqref{a3-eq} and \eqref{a2'-eq},
with $R_h=\tfrac{r}{2}-R2^{-h}$:
\begin{gather}
  \supp{\cal F}(a^{k-h}(x,D)u_k)\subset
  \bigl\{\,\xi\bigm| 
  R_h2^k\le|\xi|\le \tfrac{5R}{4} 2^k\,\bigr\},
  \label{supp1-eq}  \\
  \supp{\cal F}(a_k(x,D)u^{k-h})\subset
  \bigl\{\,\xi \bigm| 
  R_h2^k\le|\xi|\le \tfrac{5R}{4} 2^k\,\bigr\},
  \label{supp3-eq}  \\
  \supp{\cal F}\big(a_k(x,D)(u^{k-1}-u^{k-h})+(a^k-a^{k-h})(x,D)u_k\big)\subset
  \bigl\{\,\xi\bigm| |\xi|\le 2R 2^k\,\bigr\}.
  \label{supp2-eq}
\end{gather}
Such spectral corona and ball properties
have been known since the 1980's (e.g.\ \cite[(5.3)]{Y1}) 
although they were verified then only for
elementary symbols $a(x,\eta)$, in the sense of Coifman and
Meyer~\cite{CoMe78}. However, this is now a redundant restriction because of 
the \emph{spectral support rule}, which for 
$u\in {\cal F}^{-1}{\cal E}'(\Rn)$ states that
\begin{equation}
     \supp{\cal F}(a(x,D)u)\subset
\bigl\{\,\xi+\eta \bigm| (\xi,\eta)\in \supp{\cal F}_{x\to\xi\,} a,\ 
     \eta\in \supp{\cal F} u \,\bigr\},
  \label{Sigma-eq}
\end{equation}
A short proof of this can  be found in \cite[App.~B]{JJ10tmp}
(cf.\ \cite{JJ05DTL,JJ08vfm,JJ10tmp} for the full version).
Since \eqref{supp1-eq}--\eqref{supp2-eq} follow easily from \eqref{Sigma-eq},
as shown in \cite{JJ05DTL,JJ10tmp}, details are omitted. 

The novelty in relation to pointwise estimates is that the summands in the
decomposition \eqref{a1-eq}--\eqref{a3-eq} can be controlled thus:
for $a^{(1)}_\psi(x,D)u$ the fact that $k\ge h\ge 2$ allows the corona
condition of Corollary~\ref{Fa-cor} to be fulfilled for $\Theta_0=r/2$ and
$\Theta_1=R$ (i.e.\ the auxiliary function there is $1$ on $\supp\hat u$),
so \eqref{a1-eq} and the factorisation inequality simply give the first estimate:
\begin{equation}
  |a^{k-h}(x,D)u_k(x)|\le F_{a^{k-h}}(N,R2^k;x)u_k^*(N,R2^k;x)
  \le cp(a)(R2^k)^d u_k^*(x).
  \label{akh-ineq}
\end{equation}
Hereby the convolution estimate $p(a^{k-h})\le \nrm{{\cal F}^{-1}\psi}{1} p(a)$ 
is utilised to get a constant independent of $k$.

In $a^{(2)}_\psi(x,D)u$ the terms may be treated similarly:
in \eqref{a2'-eq} it is for $k\ge 1$ clear that $(a^k-a^{k-h})(x,D)u_k$ 
only requires the constant to have $\nrm{{\cal F}^{-1}(\psi-\psi(2^h\cdot))}{1}$
as a factor instead of $\nrm{{\cal F}^{-1}\psi}{1}$, cf.\ the above; 
while for $k=0$ it may just be
increased by a fixed power of $R$ using the full generality of Corollary~\ref{Fa-cor}.
The remainders in \eqref{a2'-eq} have $k>0$ and can be written as in
\eqref{a2-eq}. Hence one obtains the second estimate:
\begin{multline}
 |(a^k-a^{k-h})(x,D)u_k(x)+a_k(x,D)[u^{k-1}-u^{k-h}](x)|
\\
 \le F_{a^k-a^{k-h}}(N,R2^k;x)u_k^*(N,R2^k;x)+
      \sum_{l=1}^{h-1} F_{a_k}(N,R2^{k-l},x)u_{k-l}^*(N,R2^{k-l};x)
\\
 \le 
      cp(a) (R2^k)^d \sum_{l=0}^{h-1} 2^{-ld} u_{k-l}^*(N,R2^{k-l};x).
  \label{a2-ineq}
\end{multline}
Here the sum over $l$ is harmless, because the number of terms is independent of $k$.

The improved asymptotics of Corollary~\ref{Fa0-cor} come into play as
reinforcements for the series for $a^{(3)}_\psi(x,D)u$. Indeed, for $Q=2^j$ the
first part of \eqref{a3-eq} gives, for $M>0$, the third estimate
\begin{equation}
\begin{split}
    |a_j(x,D)u^{j-h}(x)| &\le \sum_{k=0}^{j-h} |a_j(x,D)u_k(x)| 
   \le \sum_{k=0}^{j} F_{a_j}(N,R2^k;x)u_k^*(N,R2^k;x)
\\
  &\le c'_M p(a) 2^{-jM} \sum_{k=0}^j  (R2^k)^{d+\delta M}u^*_k(N,R2^k;x).
\end{split}
  \label{aj-ineq}
\end{equation}
Here the number of terms on the right-hand side depends on $j$, but
this is manageable due to $2^{-jM}$, which serves as a summation factor. 
Altogether this proves

\begin{theorem} \label{a123-thm}
  For each symbol $a$ in $S^d_{\rho,\delta}$, $0\le\delta<\rho\le1$,
the paradifferential decomposition \eqref{a123-eq} is valid with the terms in
\eqref{a1-eq}--\eqref{a3-eq}
having the spectral relations \eqref{supp1-eq},\eqref{supp3-eq}, \eqref{supp2-eq}
and the pointwise estimates \eqref{akh-ineq}, \eqref{a2-ineq}, \eqref{aj-ineq}.
\end{theorem}

Not surprisingly, Theorem~\ref{a123-thm}  
yields boundedness in several scales. Perhaps this is most
transparent for the Besov spaces $B^{s}_{p,q}(\Rn)$. These generalise both 
the Sobolev spaces
$H^s(\Rn)$ and the H{\"o}lder spaces $C^s(\Rn)$ (with $0<s<1$, cf
Proposition~\ref{Cs*-prop}) as
\begin{equation}
  H^s=B^s_{2,2},\qquad C^s=B^s_{\infty ,\infty }.
  \label{HCB-eq}
\end{equation}
The spaces $B^{s}_{p,q}$ are for $s\in \R$, $p,q\in \,]0,\infty ]$
defined by means of the Littlewood--Paley decomposition as
the $u\in {\cal S}'$ for which the following (quasi-)norm is finite,
\begin{equation}
  \Nrm{u}{B^{s}_{p,q}}=(\sum_{j=0}^\infty
2^{sjq}\Nrm{\varphi(2^{-j}D)u}{p}^q)^{1/q};
  \label{Bspq-id}
\end{equation}
hereby the norm in $\ell_q$ should be read as the supremum over $j$ for
$q=\infty $. (Often a specific choice of the function $\psi$ is stipulated,
but this is immaterial as they all lead to equivalent norms on the spaces).
For $p,q\in [1,\infty ]$ the space $B^{s}_{p,q}$ is a Banach space.
Note that the first part of \eqref{HCB-eq} follows easily from
\eqref{Bspq-id}; cf.\ \cite{H97,T2} for the second.

Now Theorem~\ref{a123-thm} gives the following continuity result:
\begin{theorem}
  \label{Bspq-thm}
When $a(x,\eta)$ belongs to $S^d_{1,\delta}(\Rn\times \Rn)$ and $0\le
\delta<1$, then 
\begin{align}
  a(x,D)&\colon H^{s+d}(\Rn)\to H^{s}(\Rn)
  \label{aHs-eq}
\\
  a(x,D)&\colon B^{s+d}_{p,q}(\Rn)\to B^{s}_{p,q}(\Rn)
  \label{aBspq-eq}
\end{align}
is continuous for every $s\in\R $, $0<p\le \infty $, $0<q\le \infty $.
\end{theorem}

\begin{proof}
Taking $L_p$- and $\ell_q$-norms on both sides of
\eqref{akh-ineq}, Theorem~\ref{max-thm} gives  for $N>n/p$,
\begin{equation}
  (\sum_{k=0}^\infty 2^{skq}\Nrm{a^{k-h}(x,D)u_k}{p}^q)^{1/q}
 \le c p(a)  (\sum_{k=0}^\infty 2^{(s+d)kq}\Nrm{u_k}{p}^q)^{1/q} 
  =c p(a)\nrm{u}{B^{s+d}_{p,q}}.
\end{equation}
Because of the dyadic corona property \eqref{supp1-eq}, 
the above estimate implies convergence of $a^{(1)}_\psi(x,D)u=\sum a^{k-h}(x,D)u_k$
to an element in $B^{s}_{p,q}$, the norm of which is 
estimated by the right-hand side (this is well known, cf.\ \cite{Y1},\cite[2.3.2]{RuSi96}
or \cite{JJ10tmp}). So for $m=1$,
\begin{equation}
  \nrm{a^{(m)}_{\psi}(x,D)u}{B^{s}_{p,q}}\le c'' p(a)\nrm{u}{B^{s+d}_{p,q}}.
  \label{a(m)-ineq}
\end{equation}

The contribution $a^{(3)}(x,D)$ in \eqref{a123-eq} is treated similarly,
except for the sum over $k$. This is handled with a small lemma, namely
$\sum_{j=0}^\infty 2^{sjq}
(\sum_{k=0}^j |b_k|)^q\le c \sum_{j=0}^\infty 2^{sjq}|b_j|^q$, valid for all 
$b_j\in \C$ and $0<q\le \infty $ provided $s<0$; cf.\ \cite{Y1}.
Thus \eqref{aj-ineq} implies
\begin{equation}
  \begin{split}
 \sum_{j=0}^\infty 2^{sjq}   \Nrm{a_j(x,D)u^{j-h}}{p}^q
&\le \sum_{j=0}^\infty 2^{(s-M)jq} (\sum_{k=0}^{j} 
      cp(a)2^{k({d+\delta M})} \Nrm{u_k^*(N,R2^k;\cdot )}{p})^q
\\
&\le  cp(a)^q \sum_{j=0}^\infty 2^{(s+d-(1-\delta)M)jq} \Nrm{u_j}{p}^q
\\
&= cp(a)^q\nrm{u}{B^{s+d-(1-\delta)M}_{p,q}}
  \end{split}
\end{equation}
provided $M>0$ and $M>s$. This implies \eqref{a(m)-ineq} for $m=3$.

For $a^{(2)}(x,D)u$ the estimate is a little simpler, for in \eqref{a2-ineq}
one only needs to apply norms of $L_p$ and $\ell_q$ with respect to $x$ and $k$,
respectively, and use the (quasi-)triangle inequality.
Because \eqref{supp2-eq} is a dyadic ball property, the resulting estimate gives
\eqref{a(m)-ineq} with $m=2$ only in case $s>\max(0,\fracnp-n)$. 
But then, via \eqref{a123-eq}, this shows \eqref{aBspq-eq}.

However, one can reduce to such $s$ by writing 
$a(x,D)=\Lambda^t(\Lambda^{-t} a(x,D))$ with $t=2|s|+1$ (or
$t=2|s|+1+\fracnp-n$ if $0<p<1$), 
for $\Lambda^t=\OP((1+|\eta|^2)^{t/2})$ is of order 
$t$ in the $B^{s}_{p,q}$-scale. 
This shows \eqref{aBspq-eq}, hence \eqref{aHs-eq} as a special case.
\end{proof}

Theorem~\ref{Bspq-thm} is well established, of course. 
For example \cite[Thm.~18.1.13]{H} or \cite[Thm.~3.6]{SRay91} 
gives the $H^s$-part with a classical reduction to
Schur's lemma. The present
proof should be interesting because it combines 
Littlewood--Paley theory with the factorisation inequality etc.

The flexibility of this method is apparent from the fact that
it extends \emph{at once} to the $B^{s}_{p,q}$ 
with arbitrary $p,q\in \,]0,\infty ]$.
The previous proofs for $B^s_{p,q}$ in e.g.\ \cite{Bou82,Y1} are cumbersome due
to the use of elementary symbols and multiplier results.

\section{The case of type $1,1$-operators}   
  \label{t11-sect}

The above results carry over to type $1,1$-operators, ie to $\delta=1$ with almost no changes. 
Previously type $1,1$-operators have  been treated in fundamental contributions
of Ching~\cite{Chi72}, Stein 1972 (cf.~\cite{Ste93}), 
Parenti and Rodino~\cite{PaRo78},  Meyer~\cite{Mey81},
Bourdaud~\cite{Bou83,Bou88}, H{\"o}rmander~\cite{H88,H89,H97} and Torres~\cite{Tor90}.

The present methods were in fact developed for such operators, which
emphasizes the efficacy of pointwise estimates. 
But since the topic is specialised, only brief remarks on the outcome will be given here.

\bigskip

The reader may consult \cite{JJ08vfm,JJ10tmp} for a review of 
\emph{operators} of type $1,1$ and
a systematic treatment. Here it suffices to recall from \cite{JJ08vfm} that
for $a\in S^d_{1,1}(\Rn\times \Rn)$ the identity \eqref{aPsi-eq} is 
used as the \emph{definition:}
when the limit there exists in ${\cal D}'(\Rn)$ and is independent of $\psi$,
then $u$ belongs to the domain $\in D(a(x,D))$ 
and the action of $a(x,D)$ on $u$ is set equal to
the limit in \eqref{aFE-eq}; cf.\ \cite{JJ08vfm}.

For example, if $u$ is in ${\cal S}+{\cal F}^{-1}{\cal E}'$ the limit in \eqref{aPsi-eq} 
exists and equals the right-hand side of \eqref{aFE-eq}. 
Since the latter does not depend on $\psi$, nor on $v$, $v'$, 
one has by definition that ${\cal S}+{\cal F}^{-1}{\cal E}'\subset D(a(x,D))$,
and \eqref{aFE-eq} holds.
Cf.\ \cite[Cor.~4.7]{JJ08vfm}.

Therefore the proof of Theorem~\ref{Fau*-thm}, which departs from \eqref{aFE-eq}, 
can be repeated for type $1,1$-operators: 

\begin{theorem}
The factorisation inequality $a(x,D)u(x)\le F_a(x)u^*(x)$ is valid \emph{verbatim}
for symbols $a\in S^d_{1,1}(\Rn\times\Rn)$; cf.\ Theorem~\ref{Fau*-thm}.
\end{theorem}

That $a(x,D)\colon {\cal O}_M\to {\cal O}_M$
is also true for type $1,1$-operators, but the proof of Corollary~\ref{aOO-cor} 
needs to be changed to obtain 
the decisive inclusion ${\cal O}_M\subset D(a(x,D))$
(cf.~\cite{JJ10tmp} for more details on this).

However, the proof of Corollary~\ref{aFE-cor} gives without changes 

\begin{theorem}
  \label{aFE11-thm}
For $a\in S^d_{1,1}$, $d\in\R$, $p\in \,]0,\infty ]$,
$R\ge 1$ and $N>n/p$ one has
\begin{equation}
  \nrm{a(x,D)u}{p}\le C(N,R) \nrm{u}{p}
  \label{iLp-eq''}
\end{equation}
whenever $u\in L_p(\Rn)\bigcap{\cal S}'(\Rn)$
fulfils $\supp\hat u\subset \overline{B}(0,R)$.
\end{theorem}

This result is a novelty in the type $1,1$-context.
It is noteworthy because some operators in
$\OP(S^0_{1,1})$ are unbounded on $L_p$, even for $p=2$,
by a construction of Ching \cite{Chi72}\,---\,and therefore pointwise
estimates seem indispensable for Theorem~\ref{aFE11-thm}.

It is also straightforward to see that one has

\begin{theorem}  \label{Fa11-thm}
  For the symbol factor $F_a(x)$, the estimates by integrals in 
Theorem~\ref{Fa-thm} are valid \emph{verbatim} for $\delta=1=\rho$. 
Similarly the estimates by symbol seminorms
in Corollaries~\ref{Fa-cor} and \ref{Fa0-cor} carry over. In particular
the corona condition yields
\begin{equation}
  F_a(N,R;x)={\cal O}(R^d),\qquad  F_{a_Q}(N,R;x)={\cal O}(Q^{-M}R^{d+M})
  \label{Oa11-eq}
\end{equation}
for $R\to\infty$  and $Q\to\infty$ (fixed $R$), respectively.
\end{theorem}

Moreover, the paradifferential decomposition in Section~\ref{corona-sect} is
unchanged, although for type $1,1$-operators it has to be made for arbitrary
$\psi$ because of their definition.

As a difference it holds for type $1,1$-operators that the series for $a^{(2)}(x,D)u$
in \eqref{a2-eq} converges if and only if $u\in D(a(x,D))$. This results at once
from the fact that
the series for $a^{(1)}(x,D)u$ and $a^{(3)}(x,D)u$ in \eqref{a1-eq}, \eqref{a3-eq}
converge for every $u\in {\cal S}'$, which was proved in \cite{JJ10tmp} 
by combining a lemma of Coifman and Meyer \cite[Ch.~15]{CoMe97} 
with the pointwise estimates summed up in Theorem~\ref{Fa11-thm}.

Theorem~\ref{a123-thm} also carries over to
arbitrary type $1,1$-operators, whence the boundedness in 
Theorem~\ref{Bspq-thm} does so for large $s$:
\begin{theorem}
  \label{a11Bspq-thm}
For each $a$ in $S^d_{1,1}(\Rn\times \Rn)$ the operator is continuous
($p,q\in\,]0,\infty]$)
\begin{align}
  a(x,D)&\colon H^{s+d}(\Rn)\to H^{s}(\Rn),\quad\text{ for $s>0$},
  \label{a11Hs-eq}
\\
  a(x,D)&\colon B^{s+d}_{p,q}(\Rn)\to B^{s}_{p,q}(\Rn), 
  \quad\text{ for $s>\max(0,\fracnp-n)$}.
  \label{a11Bspq-eq}
\end{align}
\end{theorem}

The proof is the same,  
except that the lift operator $\Lambda^t$ is redundant.
The boundedness was essentially shown in \cite{JJ05DTL}, 
though the formal definition of type $1,1$-operators first appeared in \cite{JJ08vfm}.

However, H{\"o}rmander's condition in \cite{H88,H89,H97},
that ${\cal F}_{x\to\xi}a(x,\eta)$ be
small along the twisted diagonal $\xi=-\eta$, allows the conditions on $s$ in 
\eqref{a11Hs-eq} and \eqref{a11Bspq-eq} to be removed. 

Indeed, in terms of a specific localisation to the 
twisted diagonal, namely the symbol
$a_{\chi,\varepsilon}(x,\eta)={\cal F}_{x\to\xi}^{-1}
(\chi(\xi+\eta,\varepsilon\eta){\cal F}_{x\to\xi}a(x,\eta))$
defined in \cite{H88,H89,H97} for a suitable
$\chi\in C^\infty$ supported where $|\eta|>|\xi|$, $|\eta|>1$,
H{\"o}rmander introduced the fundamental  
condition that for every $\sigma>0$ there is an estimate for $\varepsilon>0$:
\begin{equation}
  \sup_{R>0,\; x\in \Rn}R^{-d}\big(
  \int_{R\le |\eta|\le 2R} |R^{|\alpha|}D^\alpha_{\eta}a_{\chi,\varepsilon}
  (x,\eta)|^2\,\frac{d\eta}{R^n}
  \big)^{1/2}
  \le c_{\alpha,\sigma} \varepsilon^{\sigma+n/2-|\alpha|}.
  \label{Hsigma-eq}
\end{equation}
This is first of all interesting because of the obvious similarity with the 
Mihlin--H{\"o}rmander type estimates of the symbol factor in 
Theorems~\ref{Fa-thm} and \ref{Fa11-thm}. 

Secondly, for $a(x,\eta)$ fulfilling \eqref{Hsigma-eq} 
the conditions on $s$ in Theorem~\ref{a11Bspq-thm} were removed in \cite{JJ10tmp} 
(with an arbitrarily small loss of smoothness for $p<1$).
The proof consisted in a refinement of that of Theorem~\ref{a11Bspq-thm}, 
in which the necessary improvements for $a^{(2)}(x,D)u$ were obtained by skipping 
\eqref{Oa11-eq} and controlling the $F_a$-estimates of Mihlin-H{\"o}rmander type  
directly in terms of H{\"o}rmander's condition \eqref{Hsigma-eq}.

Furthermore, Theorem~\ref{Fa11-thm} was used in \cite{JJ10tmp} 
to bridge the gap between the general Littlewood--paley theory, 
i.e.\ the paradifferential splitting \eqref{a123-eq}, 
and symbols fulfilling \eqref{Hsigma-eq}, again
by controlling the Mihlin--H{\"o}rmander type estimates in terms of \eqref{Hsigma-eq}.
Thus it was proved explicitly in \cite{JJ10tmp} that the series for
$a^{(2)}(x,D)u$, for $u\in{\cal S}'$, converges in the topology of ${\cal S}'(\Rn)$ 
whenever \eqref{Hsigma-eq} holds.

In this connection, it deserves to be mentioned that the approach of Marschall,
recalled in Remark~\ref{Marschall-rem}, would be insufficient, since application of 
\eqref{Marschall-ineq} to the terms in $a^{(2)}(x,D)u$ would result in estimates involving
$Mu_k$: for general $u\in{\cal S}'$ the Hardy--Littlewood maximal function $Mu_k$ 
would not be finite, unlike $u^*_k(N,R;x)$ that is so for all sufficiently large $N$.

The boundedness results extend to the Sobolev spaces $H^s_p=\Lambda^{-s}L_p$, with 
$s\in\R$ and $1<p<\infty$, cf.\ \cite{JJ08vfm}, or more generally to the
Lizorkin--Triebel scale $F^{s}_{p,q}$ with $0<p<\infty$, $0<q\le\infty$; 
cf.\ \cite{JJ10tmp}.
The proofs follow the lines indicated above.

It would be outside of the topic here to give the full statements, so
the reader is referred to \cite{JJ08vfm,JJ10tmp} for more details on the
results for operators of type $1,1$.

\section{Final remarks}   \label{final-sect}
Pointwise estimates in terms of the maximal function $u^*$ were crucial for the
author's work on type $1,1$-operators \cite{JJ10tmp} and developed for that purpose, 
but they were used there with
only brief explanations. A detailed presentation has been postponed to the
present paper, because the techniques should be of interest in their own
right. This is illustrated by the proofs of
Corollary~\ref{aOO-cor},
Theorem~\ref{Bspq-thm} and Theorem~\ref{aFE11-thm} e.g.

\providecommand{\bysame}{\leavevmode\hbox to3em{\hrulefill}\thinspace}

\end{document}